\documentclass[11pt]{article}

\input{WAAGmac}  

\title{\vskip-1.0em\sc Weak and cyclic amenability for Fourier algebras of connected Lie groups}
\author{\sc Y.Choi and M. Ghandehari}

\date{4th March 2014}

\usepackage{graphics}
\usepackage{sectsty}
\allsectionsfont{\sffamily}

\usepackage[usenames]{color}
\newcommand{\dt}[1]{\textcolor{Bittersweet}{\sf#1}}
\renewcommand{\sch}{{\operatorname{Sch}}}
\newcommand{\para}[1]{\paragraph{#1.}}  

\begin{document}

\maketitle

\begin{abstract}
Using techniques of non-abelian harmonic analysis, we construct an explicit, non-zero cyclic derivation on the Fourier algebra of the real $ax+b$ group. In particular this provides the first proof that this algebra is not weakly amenable.
Using the structure theory of Lie groups, we deduce that the Fourier algebras of connected, semisimple Lie groups also support non-zero, cyclic derivations and are likewise not weakly amenable. Our results complement earlier work of Johnson (JLMS, 1994), Plymen (unpublished note) and Forrest--Samei--Spronk (IUMJ 2009).

As an additional illustration of our techniques, we construct an explicit, non-zero cyclic derivation on the Fourier algebra of the reduced Heisenberg group, providing the first example of a connected nilpotent group whose Fourier algebra is not weakly amenable.

\end{abstract}


\begin{section}{Introduction}
\subsection*{Background context and history}
The study of derivations from Banach algebras into Banach bimodules has a long history. In many cases, where the algebra consists of functions on some manifold or well-behaved subset of Euclidean space, and the target bimodule is symmetric, continuous derivations can be constructed by taking weighted averages of derivatives of functions. This leads to examples where the algebra does not admit any non-zero continuous `point' derivations, yet admits a non-zero continuous derivation into \emph{some} symmetric Banach bimodule; this is often a manifestation of some kind of vestigial analytic structure or H\"older continuity of functions in the algebra. Commutative Banach algebras which admit no non-zero, continuous derivations into \emph{any} symmetric Banach bimodule are said to be \dt{weakly amenable}: the terminology was introduced by W. G. Bade, P. C. Curtis Jr.~and H. G. Dales in \cite{BCD_WA}, where they studied some key examples in detail.

One natural class of function algebras not considered in \cite{BCD_WA} is the class of Fourier algebras of locally compact groups, first defined in full generality by P. Eymard in \cite{eymard64}. Fourier algebras never admit any non-zero, continuous point derivations: this follows from e.g.~\cite[(4.11)]{eymard64}. Moreover, if $G$ is a locally compact abelian group, its Fourier algebra $\FA(G)$ is isomorphic via the classical Fourier transform to the convolution algebra $L^1(\widehat{G})$, and so by results of B.~E.~Johnson it is \emph{amenable}, hence weakly amenable (\cite[Proposition 8.2]{BEJ_CIBA}). 

Now a recurring theme in abstract harmonic analysis, and the study of Fourier algebras in particular, is the hope that known results for locally compact abelian groups can be generalized in a natural way to the class of locally compact amenable groups. It was therefore something of a surprise when Johnson, in \cite{BEJ_AG}, constructed a non-zero bounded derivation from $\FA(\SO_3(\Real))$ to its dual. Subsequently, using the structure theory of semisimple Lie algebras, R. J. Plymen observed \cite{Plymen_AG} that Johnson's construction can be transferred to yield non-zero continuous derivations on $\FA(G)$ for any non-abelian, compact, connected Lie group. This was extended further by B. E. Forrest, E. Samei and N. Spronk in \cite{FSS_WAAG}, who showed using structure theory for compact groups that Plymen's result remains valid if one drops the word ``Lie''.

The non-compact, connected case has received relatively little attention. The articles \cite{Plymen_AG} and \cite{FSS_WAAG} ultimately rely on locating closed copies of $\SO_3(\Real)$ or $\SU_2(\Cplx)$ inside the group in question, and then transporting Johnson's derivation along the corresponding restriction homomorphism of Fourier algebras. Indeed, as far as the present authors are aware, all results to date which show that $\FA(G)$ fails to be weakly amenable only work for those $G$ containing compact, connected, non-abelian subgroups. This has left open several natural examples, such as the ``real $ax+b$ group'' (to be defined precisely in Section~\ref{s:ax+b}), or certain semisimple Lie groups such as $\SL_2(\Real)$ and its covering groups.

\subsection*{Pr\'ecis of our results}
Our first main result is the construction of an explicit, non-zero, continuous derivation on the Fourier algebra of the real $ax+b$ group, by which we mean the connected component of the affine group over $\Real$. Note that this group does not contain any non-trivial connected compact subgroups, let alone ones which are non-abelian.
As with the example constructed in~\cite{BEJ_AG}, our derivation is also \emph{cyclic}, that is, it satisfies the identity $D(a)(b)=-D(b)(a)$ for all $a$ and $b$ in the algebra.

The construction of our example, and the verification that it has the right properties, take up all of Section~\ref{s:ax+b}. Our approach is different from that of \cite{FSS_WAAG}, and proceeds by exploiting orthogonality relations for coefficient functions of irreducible representations. Thus, to motivate our approach, we preface our construction with an expository section (Section~\ref{s:su2}) where we revisit Johnson's original construction in \cite{BEJ_AG} and give a slightly different presentation of how the desired norm bounds can be obtained.
It seems to have gone under-emphasised that the construction is very concrete and can be explained independently of other, harder results in~\cite{BEJ_AG}.

Our original motivation for considering the $ax+b$ group is that although it is non-compact, it shares one feature with compact groups: the left regular representation can be decomposed as a \emph{direct sum} of unitary representations that are quasi-equivalent to \emph{irreducible} ones.
 (In more precise technical language: this group is an \dt{AR-group}.)
It follows that its Fourier algebra admits a convenient direct sum decomposition analogous to those of compact groups, where functions in the summands satisfy orthogonality relations.

For more general connected groups, we cannot expect to have such detailed knowledge of the unitary dual (and for groups that are not of Type~I, we cannot hope for a well-behaved decomposition of the Fourier algebra in terms of irreducible representations). Nevertheless, by combining our results for the $ax+b$ group with some structure theory for connected Lie groups, we show that $\FA(G)$ fails to be weakly amenable if $G$ is a semisimple, connected Lie group.
Further structure theory allows us to deduce that if $G$ is a simply connected Lie group whose Fourier algebra is weakly amenable, then $G$ must be solvable (and satisfy some extra conditions).
Details are given in Section~\ref{s:lie-struct}.

Since the results just mentioned (and those of \cite{FSS_WAAG}) tell us nothing about the cases of nilpotent Lie groups, in 
Section~\ref{s:reduced-H} we outline how one can use techniques similar to those of Section~\ref{s:ax+b} to construct an explicit, non-zero, continuous and cyclic derivation on the Fourier algebra of the reduced Heisenberg group $\bbH_r$. This group may be viewed as the quotient of the usual $3$-dimensional Heisenberg group $\bbH$ by a discrete subgroup of its centre. In forthcoming work we will address the question of weak amenability of $\FA(\bbH)$, using the Plancherel theorem for $\bbH$ as a substitute for orthogonality relations.

\subsection*{Acknowledgements}
The first author was supported by NSERC Discovery Grant 402153-2011. He thanks E. Samei and N. Spronk for several enlightening conversations over the years about the article~\cite{BEJ_AG}, and various contributors to the {\it MathOverflow}\/ website for tolerating simple-minded questions about Lie groups and their structure theory.
The second author thanks K.~F. Taylor for discussions on square-integrable representations in the non-unimodular setting.

The work presented here was done while both authors worked at the University of Saskatchewan, and we gratefully acknowledge the support of the Department of Mathematics and Statistics during trying times~\cite{Howe_VOX}. Finally, we would like to thank the referee for several valuable suggestions, in particular the arguments in the appendix.
\end{section}

\begin{section}{Preliminaries}
\subsection{Weak and cyclic amenability for commutative Banach algebras}\label{ss:weak-cyc-amen}
We assume familiarity with the basic definitions and properties of Banach algebras and Banach modules over them. We say that a bimodule $M$ over a given algebra $A$ is \dt{symmetric} if $a\cdot m = m\cdot a$ for all $a\in A$ and $m\in M$. If $A$ is an algebra and $M$ is an $A$-bimodule, then a linear map $D:A\to M$ is said to be a \dt{derivation} if it satisfies the Leibniz identity $D(ab)=a\cdot D(b) + D(a)\cdot b$ for all $a,b\in A$.

In the case where $A$ is a Banach algebra and $M=A^*$ is its dual, we say that a derivation $D:A\to A^*$ is \dt{cyclic} if it furthermore satisfies $D(a)(b)=-D(b)(a)$ for all $a,b\in A$. (Motivation for considering cyclic derivations lies beyond the scope of the present paper, but some idea is given by the results and remarks of \cite{Gron_cyc-amen}.)

The following definitions are due to Bade, Curtis and Dales \cite{BCD_WA} and Gr{\o}nb{\ae}k \cite{Gron_cyc-amen}. \emph{We will only give definitions valid for commutative Banach algebras}: for the appropriate generalizations to the non-commutative setting, the reader may consult the wider literature on weak and cyclic amenability.

\begin{dfn}
Let $A$ be a commutative Banach algebra. We say that $A$ is \dt{weakly amenable} if there is no non-zero, continuous derivation from $A$ to any symmetric Banach $A$-bimodule. We say that $A$ is \dt{cyclically amenable} if there is no non-zero, continuous, cyclic derivation from $A$ to $A^*$.
\end{dfn}

\begin{rem}
Clearly, weak amenability implies cyclic amenability. The converse is not true in general: any singly generated Banach algebra is cyclically amenable, as can be seen by looking at the values a continuous cyclic derivation must take on powers of the generator, while there are many examples of singly generated Banach algebras --- even finite dimensional ones --- that support bounded, non-zero point derivations, and hence are very far from being weakly amenable.
\end{rem}

The following well-known observations will be used in Sections~\ref{s:su2} and~\ref{s:lie-struct}, see also Proposition~\ref{p:WA-hered} below. If $A, B$ are commutative Banach algebras and $\theta:A\to B$ is a continuous homomorphism with dense range, then continuous, non-zero derivations on $B$ can be pulled back along $\theta$ to give continuous, non-zero derivations on $A$. Consequently, if $A$ is weakly amenable then so is~$B$. Since cyclic derivations pull back to cyclic derivations, if $A$ is cyclically amenable then so is~$B$.

\begin{rem}
We repeat that the definitions of weak and cyclic amenability are less straightforward for non-commutative Banach algebras. In general, quotients of a cyclically amenable noncommutative Banach algebra need not be cyclically amenable.
\end{rem}

\subsection{Coefficient functions and the Fourier algebra}
To make the present paper more accessible to those whose background is in Banach algebras rather than abstract harmonic analysis, we use this subsection to review some background results and terminology. Our approach is heavily influenced by work of G.~Arsac \cite{Arsac};
see also the master's thesis of C.~Zwarich~\cite{Zwarich_MSc} for a good summary and exposition.

We start in some generality. Let $G$ be a locally compact group. A \dt{continuous unitary representation} of  $G$ on ${\cH}$ is a group homomorphism $\pi$ of $G$
into the group of unitary operators ${\cal U}({\cH})$ which is WOT-continuous, i.e.
for every vector $\xi$ and $\eta$ in ${\cH}$, the function
\[ \xi*_\pi\eta:G\rightarrow \Cplx, \quad g\mapsto \pair{ \pi(g)\xi }{ \eta } \]
is continuous. (For unitary representations, WOT-continuity is equivalent to SOT-continuity; the latter is often used instead in the definition.)
We denote unitary equivalence of two representations $\pi$ and $\sigma$ by $\pi\simeq\sigma$. The collection of irreducible, continuous unitary representations modulo unitary equivalence is denoted by~$\widehat{G}$.

Functions of the form $\xi*_\pi\eta$, for vectors $\xi$ and $\eta$ in ${\cH}$, are called the \dt{coefficient functions}
of $G$ \dt{associated }with the representation~$\pi$.
Following \cite{eymard64}, we denote by $\FS(G)$ the set of all the coefficient functions of~$G$. This is in fact a subalgebra of $C_b(G)$, known as the \dt{Fourier--Stieltjes algebra} of $G$.
%
Moreover we may identify $\FS(G)$ with the dual Banach space of the full group $\Cst$-algebra $\Cst(G)$. (The idea, briefly, is that a coefficient function $\xi*_\pi\eta$ may be identified with the functional $a\mapsto \pair{\pi(a)\xi}{\eta}$.)
Equipped with this norm and the algebra structure inherited from $C_b(G)$, $\FS(G)$ becomes a Banach algebra.

Let $\lm$ denote the left regular representation of $G$ on $L^2(G)$.
Eymard \cite[Ch.~3]{eymard64} showed that the $\norm{\cdot}_{\FS(G)}$-closure of the algebra $(C_c\cap\FS)(G)$ coincides with the closed subspace generated by coefficient functions associated to~$\lm$. This closed subalgebra, denoted by $\FA(G)$, is the \dt{Fourier algebra} of $G$.
In fact, every element of $\FA(G)$ can be realized as a coefficient function associated to $\lm$, and we have
\begin{equation}\label{eq:A(G)-as-A_lambda}
\norm{u}_{\FA(G)}=\inf\{\norm{\xi}_2\norm{\eta}_2: u=\xi*_\lambda\eta\}.
\end{equation}
In many sources, the Fourier algebra of $G$ is \emph{defined} to be the set of coefficient functions of~$\lm$, and then shown to be an algebra using the ``absorption'' properties of~$\lm$. (See \cite[\S4.1]{Zwarich_MSc} for a quick exposition of this approach.)

\begin{rem}
If $G$ is a locally compact \emph{abelian} group, one can identify $\FA(G)$ and $\FS(G)$ with the $L^1$-algebra and the measure algebra, respectively, of~$\widehat{G}$.
\end{rem}

Now let $H$ be a closed subgroup of $G$, and let $\imath^*:C_0(G)\to C_0(H)$ be the restriction homomorphism. 
It turns out that $\imath^*$ maps $\FA(G)$ contractively onto $\FA(H)$: this is originally due to C.~Herz, but an approach using spaces of coefficient functions was given by G. Arsac~\cite{Arsac}. (A fairly self-contained account of Arsac's approach can be found in \cite[\S4]{Zwarich_MSc}.)
Recalling our earlier remarks about derivations on Banach algebras, we therefore have the following well-known result, whose proof we omit.

\begin{prop}\label{p:WA-hered}
Let $G$ be a locally compact group and $H$ a closed subgroup. If $\FA(H)$ is not weakly amenable, then $\FA(G)$ is not weakly amenable. If $\FA(H)$ is not cyclically amenable, then $\FA(G)$ is not cyclically amenable.
\end{prop}

\para{Spaces of coefficient functions associated to a fixed representation}
Proofs of the results stated here may be found in Arsac's thesis \cite{Arsac}; see also \cite[\S3.5]{Zwarich_MSc}.

Let $\pi$ be a continuous unitary representation of $G$ on a Hilbert space ${\cH}_\pi$. We define $\FA_\pi(G)$ to be the closed subspace of $\FS(G)$
generated by the coefficient functions of $G$ associated with $\pi$, i.e.
\[ \FA_\pi=\overline{\lin}^{\norm{\cdot}_{\FS(G)}}\{\xi*_\pi\eta:\xi,\eta\in {\cH}_\pi\}. \]
\begin{rem}\label{r:A_pi_of_irred}
There is a natural quotient map of Banach spaces  $\cH_\pi \ptp \overline{\cH_\pi} \to \FA_\pi$, where $\ptp$ denotes the projective tensor product of Banach spaces.
 (See \cite[Th\'eor\`eme~2.2]{Arsac}.) Usually this map is not injective, but it will be if $\pi$ is irreducible (this follows by taking the adjoint of this map and using Schur's lemma).
\end{rem}

In particular, $\FA_\pi(G)$ consists of all coefficient functions $u$ that can be written in the form
$u=\sum_{i=1}^\infty \xi_n*_\pi\eta_n$\/,
where $\xi_n$ and $\eta_n$ belong to ${\cH}_\pi$ and $\sum_{i=1}^\infty\norm{\xi_i}\norm{\eta_i}<\infty$.
Moreover, for every $u$ in $\FA_\pi(G)$,
\[
\norm{u}_{\FS(G)}=\inf\{\sum_{i=1}^\infty\norm{\xi_i}\norm{\eta_i}: \text{ $u$ represented as above}\},
\]
and the infimum is attained. That is, given $u\in \FA_\pi$\/, we can always write it as an absolutely convergent sum $u = \sum_{n=1}^\infty \xi_n *_{\pi} \eta_n$ where $\sum_{n=1}^\infty \norm{\xi_n} \norm{\eta_n}= \norm{u}$.

An irreducible representation $\pi$ is called \dt{square-integrable} if there is some nonzero square-integrable coefficient function associated to~$\pi$. A representation $\pi$ is  square-integrable if and only if it is equivalent to a sub-representation of $\lambda$, in which case $\FA_\pi(G)\subseteq \FA_\lm(G)=\FA(G)$. Clearly all irreducible representations of a \emph{compact} group are square-integrable. Importantly for us, both the $ax+b$ group and the reduced Heisenberg group also have in some sense ``enough square-integrable representations'' that techniques used in the compact case can be adapted to handle these two non-compact groups.

\para{The case of compact groups}
Although the focus of our paper is on certain non-compact groups, our approach is informed by properties of Fourier algebras of compact groups, which we now briefly review.
When $G$ is a compact group, $\FA(G)=\FS(G)$, so $\FA_\pi(G)\subseteq \FA(G)$ for all continuous unitary representations~$\pi$. Moreover there is an $\ell^1$-direct sum decomposition of Banach spaces
\begin{equation}\label{eq:decomp_compact}
\FA(G) = \bigoplus_{\pi\in\widehat{G}} \FA_\pi(G) \iso \bigoplus_{\pi\in\widehat{G}} \cH_\pi\ptp\overline{\cH_\pi}\ .
\end{equation}
The \dt{Schur orthogonality relations} for coefficient functions of irreducible representations are as follows:
given $\pi,\sigma\in\widehat{G}$ and $\xi_1, \eta_1\in \cH_\pi$ and $\xi_2$, $\eta_2\in \cH_\sigma$, we have
\begin{equation}\label{eq:schur-orth-compact}
\pair{\xi_1*_\pi\eta_1}{\xi_2*_\sigma\eta_2}_{L^2(G)}
= \left\{ \begin{aligned}
 & 0 & \quad\text{if $\pi\not\simeq\sigma$}  \\
& \dim(\pi)^{-1} \pair{\xi_1}{\xi_2}\pair{\eta_2}{\eta_1} & \quad\text{if $\pi=\sigma$}
\end{aligned}
\right.
\end{equation}

\end{section}


\begin{section}{Revisiting Johnson's result}\label{s:su2}
In this section, to reduce congested notation we will abbreviate $\SO_3(\Real)$ and $\SU_2(\Cplx)$ to $\SOTHREE$ and $\SUTWO$ respectively.

Johnson proves in \cite[\S7]{BEJ_AG} that the Fourier algebra of $\SOTHREE$ is not weakly amenable. In fact, he shows (Theorem 7.4, {\it ibid.}\/) that this algebra supports a non-zero \emph{cyclic derivation}, in the sense of Section~\ref{ss:weak-cyc-amen}. It is the concrete construction given in the proof of this result, rather than the general machinery developed in the rest of his paper, which forms the basis for our approach.
In this section we review his construction, giving a slightly different presentation of the ideas, which will generalize in a better way to non-compact groups.

For technical reasons, we work not on $\SOTHREE$ but on its double cover $\SUTWO$. Note that the covering map $\SUTWO\to\SOTHREE$ induces an isometric inclusion of algebras $\imath:\FA(\SOTHREE)\to\FS(\SUTWO)=\FA(\SUTWO)$. Hence, to prove that neither $\FA(\SOTHREE)$ nor $\FA(\SUTWO)$ are weakly amenable, it suffices to construct a bounded derivation $D:\FA(\SUTWO)\to\FA(\SUTWO)^*$ and check that $D(\imath(f))(\imath(g))\neq 0$ for some $f,g\in\FA(\SOTHREE)$.

Given a well-behaved compact Riemannian manifold $M$, one naturally obtains derivations on $C^\infty(M)$ by taking partial derivatives along some vector field. For compact Lie groups this can be done in a down-to-earth way. We consider the case $G=\SUTWO$ and, guided by the calculations of \cite[\S7]{BEJ_AG}, make the following definitions.
For $\phi\in\Real$ let 
\[ s_\phi= \left( \begin{matrix}e^{i\phi/2} & 0 \\ 0 & e^{-i\phi/2} \end{matrix} \right) \]
and for $p\in\SUTWO$, $f\in C^1(\SUTWO)$, we define
\[ \partial_\phi f (p)= \left.\frac{\partial}{\partial \phi} f(ps_\phi)\right\vert_{\phi=0} \equiv \lim_{\phi\to 0} \frac{f(ps_\phi)-f(p)}{\phi}\ . \]
(The family $(s_\phi)_{\phi\in\Real}$ generates a maximal torus of $\SUTWO$: we do not need this fact directly, but in some sense it underlies the estimates we use later.)
Clearly $\partial_\phi$ defines a continuous derivation $C^1(\SUTWO)\to C(\SUTWO)$. Define $D_\flat: C^1(\SUTWO)\times C(\SUTWO)\to\Cplx$ by
\[  D_\flat(f,g) = \int_{\SUTWO} (\partial_\phi f) g\,d\mu \,.\]
Then $D_\flat$, viewed as a linear map $C^1(\SUTWO)\to C(\SUTWO)^*$, is a derivation. 
What is less obvious --- and in effect, what Johnson proved --- is that $D_\flat$ satisfies the inequality
\begin{equation}\label{eq:desired}
 \abs{D_\flat(f)(g) } \leq C \norm{f}_{\FA} \norm{g}_{\FA} \tag{$\dagger$}
\end{equation}
for all trigonometric polynomials $f$ and $g$, and some constant~$C$. (Here and in the rest of this section, we denote the norm on $\FA(\SUTWO)$ by $\norm{\cdot}_{\FA}$ for sake of legibility.)
From this it is routine to deduce that $D_\flat$ extends to a bounded linear map $D:\FA(\SUTWO)\to\FA(\SUTWO)^*$; since $D_\flat$ is a non-zero derivation, so is~$D$.

Johnson proves the inequality \eqref{eq:desired} using an auxiliary algebra $\FA_\gamma$ which is only well-defined when $G$ is a compact group. Since we have non-compact examples in mind we take a different approach, and consider directly the effect of $\partial_\phi$ on coefficient functions of each $\pi\in\widehat{\SU(2)}$. We find that
\[
\partial_\phi(\xi*_\pi\eta)(p)
  = \left.\frac{\partial}{\partial\phi} \pair{\pi(p)\pi(s_\phi)\xi}{\eta}\right\vert_{\phi=0} 
  = ((\sF_\pi \xi)*_{\pi}\eta )(p) 
\]
where $\sF_\pi$ is defined to be the operator  $\left.\frac{\partial}{\partial \phi} \pi(s_\phi) \right\vert_{\phi=0} \in \Bdd(\cH_\pi)$.
In particular, $\partial_\phi$ maps each coefficient space $\FA_\pi(\SUTWO)$ to itself. Taking $\ell^1$-summable linear combinations of coefficient functions, and using the Schur orthogonality relations \eqref{eq:schur-orth-compact}, we have the following:
given $\pi,\sigma\in\widehat{\SU(2)}$ and $u\in\FA_\pi(\SUTWO)$, $w\in\FA_\sigma(\SUTWO)$, then
\[ \begin{aligned}
  \int_{\SUTWO} (\partial_\phi u)\overline{w}\,d\mu & = 0 & \quad\text{if $\pi\not\simeq\sigma$} \\
\left\vert \int_{\SUTWO} (\partial_\phi u)\overline{w}\,d\mu\right\vert
 & \leq \dim(\pi)^{-1} \norm{\sF_\pi} \norm{u}_{\FA} \norm{w}_{\FA} & \quad\text{if $\pi=\sigma$} \end{aligned}  \]
Therefore, since $\FA(\SUTWO) =\bigoplus_\pi \FA_\pi(\SUTWO)$ and complex conjugation of representations is a bijection of $\widehat{\SUTWO}$, we arrive at the inequality
\[ \left\vert \int_{\SUTWO} (\partial_\phi f)g\,d\mu\right\vert
 \leq \left(\sup\nolimits_{\pi\in\widehat{\SUTWO}} \frac{\norm{\sF_\pi}}{\dim(\pi)} \right) \norm{f}_{\FA} \norm{g}_{\FA} \quad\text{for all $f,g\in\Trig(\SUTWO)$.} \]
Therefore, to show that \eqref{eq:desired} holds, we only need to prove that
\begin{equation}\label{eq:BEJ_bound}
\sup\nolimits_{\pi\in\widehat{\SUTWO}} \dim(\pi)^{-1} \norm{\sF_\pi}  <\infty.
\tag{$\ddagger$}
\end{equation}
This calculation was done in \cite[\S7]{BEJ_AG}, using the well-known representation theory of $\SUTWO$. Given $\pi\in\widehat{\SUTWO}$, put $n=\dim(\pi)-1$; then with respect to the standard choice of basis for~$\cH_\pi$, the matrix $\pi(s_\phi)$ is diagonal with entries $e^{in\phi/2}, e^{i(n-2)\phi/2}$, \dots, $e^{-in\phi/2}$. It follows that $\norm{\sF_\pi} \leq \dim(\pi)/2$ and we have the required uniform bound.

Finally, let $\pi$ be the standard representation of $\SOTHREE$ on $\Real^3$, and regard it as a representation of $\SUTWO$. Let $\xi$ be any vector in $\Real^3$ such that $\pair{\sF_\pi\xi}{\xi}\neq 0$. Then $f=\xi*_\pi\xi\in\FA(\SOTHREE)$ and we find that
\[ D(\imath(f))(\imath(\overline{f})) = \frac{1}{3} \pair{\sF_\pi\xi}{\xi}\pair{\xi}{\xi} \neq 0 \]
so that $\imath^* D$ is a non-zero, bounded derivation from $\FA(\SOTHREE^*)$ to its dual, as required.

\begin{rem}\label{r:copy-of-so3-or-su2}
We only needed knowledge of the irreducible representations of $\SUTWO$ in order to get a suitable estimate on the norms of the operators $\sF_\pi$. 
As Johnson remarks in \cite[\S7]{BEJ_AG}, the same method would work on any other compact connected Lie group~$G$, provided that one can choose a suitable homomorphism $s: \Real\to G$ for which $\sF_\pi \defeq (\partial/\partial \phi)\pi(s_\phi)\vert_{\phi=0}$ satisfies a bound analogous to~\eqref{eq:BEJ_bound}. Plymen \cite{Plymen_AG} observed that this proviso is always met for every non-abelian, compact, connected Lie group $G$. However, inspection of his argument shows that it proceeds by locating a closed subgroup of $G$ isomorphic to either $\SOTHREE$ or $\SUTWO$, and so we may appeal instead to Proposition~\ref{p:WA-hered}. This is, for instance, the approach taken in \cite{FSS_WAAG}.
\end{rem}

\begin{rem}[An approach via the Plancherel formula]\label{r:use-Plancherel-for-compact}
In \cite{BEJ_AG}, the Fourier algebra of a compact group is considered as the collection of functions on the group whose ``non-abelian Fourier series'' converges absolutely. (This is the older point of view on the Fourier algebra, predating Eymard's paper; historical details may be found in \cite[\S34]{HewRoss2}.) More precisely, given a compact group $G$ and $f\in \FA(G)$,
\[ \norm{f}_{\FA} = \sum_{\pi\in\widehat{G}} \dim(\pi) \norm{\pi(f)}_1 \]
where $\pi:L^1(G)\to\Bdd(\cH_\pi)$ is the algebra homomorphism obtained by integrating the unitary representation $\pi:G\to\cU(\cH_\pi)$, and $\norm{\cdot}_1$ denotes the trace-class norm. (Compare this with the identity \eqref{eq:decomp_compact}.)
One also has the \dt{Plancherel formula}
\[ \pair{f}{g}_{L^2(G)} = \sum_{\pi\in \widehat{G}} \dim(\pi) \Tr(\pi(f)\pi(g)^*) .\]
Now, in the case $G=\SUTWO$, it is easily verified that for any $f\in\Trig(\SUTWO)$ we have $\pi(\partial_\phi f) = \pi(f)(\sF_\pi)^*$. Combining this with the Plancherel formula and the inequality \eqref{eq:BEJ_bound}, one obtains an alternative proof of the desired inequality \eqref{eq:desired}. We will return briefly to this theme at the end of the paper.
\end{rem}

\end{section}

\begin{section}{The $ax+b$ group}\label{s:ax+b}
How can we extend or adapt the argument of Section~\ref{s:su2} to non-compact cases? There are two convenient features that we exploited when $G$ is compact: we can decompose $\FA(G)$ as a direct sum of coefficient spaces of \emph{irreducible} representations; and coefficient functions of irreducible representations satisfy explicit orthogonality relations. Groups whose Fourier algebras have the first property are called \dt{AR-groups} and it turns out that one can find connected, non-compact examples; the price one pays is that these examples are usually non-unimodular. For such groups there are generalized versions of the Schur orthogonality relations, although non-unimodularity means they are not as straightforward as in the compact case, as we shall see.

One of the simplest examples of a non-compact, connected AR-group is the so-called ``real $ax+b$ group''. To be precise,
we define it to be the group $G$ of orientation-preserving affine transformations of $\Real$, i.e. 
\[ G=\left\{\twomat{a}{b}{0}{1}: a\in\Real_+^*, b\in \Real\right\}. \]
%
Here, $\Real^*_+=(0,\infty)$ should be interpeted as
 ``the positive part of the multiplicative group of the field~$\Real$'', and it carries a natural Haar measure $t^{-1}dt$.

We may identify $G$ with  the semidirect product $\Real\rtimes \Real_+^*=\{(b,a)\colon b\in \Real, a\in \Real_+^*\}$, where $\Real_+^*$ acts on $\Real$ by multiplication. 
Recall that the left Haar measure of $G$ is (up to a constant multiple) given by $d\mu(a,b)=a^{-2}\ da\ db$, where $da$ and $db$ both denote the Haar measure of $\Real$. The modular function of $G$ is $\Delta(a,b)=\frac{1}{a}$.

\para{Outline of our construction}
For $f$ a ``suitable'' function on $G$, the new function $\madb f$ defined by
\[ (\madb f)(b,a) = -\frac{1}{2\pi i} a\frac{\partial f}{\partial b}(b,a) \qquad(b\in\Real, a\in\Real_+^*) \]
is well-defined and belongs to $C_0(G)\cap L^1(G)$.
If we now define a bilinear map $D_\flat$ (on some suitable dense subalgebra of $\FA(G)$) by
\[ D_\flat(f,g) \defeq \int_G  (\madb f)(b,a) g(b,a) \ d\mu(b,a)\; \]
then $D_\flat$ satisfies the Leibniz identity, in the sense that
\[ D_\flat(fg,h) = D_\flat(g,hf) + D_\flat(f, gh) \quad\text{for $f,g,h$ ``suitable'' functions on $G$.} \]
The operator $\madb$ is chosen in such a way that,
by using the orthogonality relations for $\FA_{\pi_{\pm}}$\/, we can obtain the upper bound
\[ \abs{ D_\flat(v)(w) } \leq \norm{v}_{\FA(G)} \norm{w}_{\FA(G)} \quad\text{when $v$ and $w$ are ``convenient'';} \]
 The orthogonality relations also show explicitly that $D_\flat$ is not identically zero. Provided that ``convenient'' functions are ``suitable'' and are dense in $\FA(G)$, we may then take the unique continuous extension of $D_\flat$ to a bounded bilinear map $D:\FA(G)\times \FA(G)\to\Cplx$. Finally we use continuity arguments to show that $D$ satisfies the same identity as $D_\flat$, but this time for all functions in $\FA(G)$ and not just the ``suitable'' ones.

\bigskip
To make this outline into a proof, we need to replace ``suitable'' and ``convenient'' by precise conditions. In doing so, the product of convenient functions might not be convenient, in which case the last part of our task --- showing that the continuous extension of $D_\flat$ is still a derivation --- is not as immediate as one might expect. We can get round this using the following lemma, which is stated in a general setting of Banach algebras and dense subspaces to show that the ideas involved are not limited to our particular example.

\begin{lem}[Continuous extensions of derivations]
\label{l:fussy-extension}
Let $A$ be a Banach algebra and let $V$ be a dense linear subspace. Let $B$ be a subspace of $A$ that contains $V$ and $V\cdot V= \{fg \colon f,g\in V\}$.

Suppose that $D_\flat: B\times B\to\Cplx$ is a bilinear map with the following properties:
\begin{newnum}
\item\label{li:bounded1} for each $w\in V$, the linear maps $D_\flat(\blank, w) :B\to\Cplx$ and $D_\flat(w,\blank): B\to\Cplx$ are $\norm{\cdot}_A$-continuous;
\item\label{li:leibniz}
$D_\flat(fg,h) = D_\flat(g,hf) + D_\flat(f,gh)$ for all $f,g,h\in V$;
\item\label{li:bounded2} there is a constant $C$ such that $\abs{D_\flat(v,w)} \leq C \norm{v}_A \norm{w}_A$ for all $v,w\in V$.
\end{newnum}
Then there is a unique continuous linear map $D:A\to A^*$ which, when viewed as a bilinear form on $A$, agrees with $D_\flat$ on $V\times V$. Moreover,
\[ D(b)(v) =D_\flat(b,v) \text{ and } D(v)(b)=D_\flat(v,b) \quad\text{for all $b\in B$ and $v\in V$,} \]
and $D$ is a derivation from $A$ to $A^*$.
\end{lem}

\begin{proof}
The existence and uniqueness of a continuous linear map $D:A\to A^*$ satisfying
\begin{equation}\label{eq:agree}
 D(v)(w) = D_\flat(v,w) \quad\text{for all $v,w\in V$} \tag{$*$}
\end{equation}
follows from \ref{li:bounded2} by standard functional analysis. Then, by \ref{li:bounded1}, if $w\in V$ then
 $D(\blank, w)$ and $D_\flat(\blank, w)$ are $\norm{\cdot}_A$-continuous maps agreeing on a $\norm{\cdot}_A$-dense subset of $B$, and hence they agree on all of~$B$. Similarly, $D(v, y)=D_\flat(v, y)$ for all $v\in V$ and $y\in B$.

Finally, to show $D:A\to A^*$ is a derivation, we must prove that
\[ D(ab)(c) = D(b)(ca) + D(a)(bc) \]
for all $a,b,c\in A$. Since $V$ is dense in $A$, it suffices by continuity to prove this for all $a,b,c\in V$. But since $V+V\cdot V\subseteq B$ and $D$ agrees with $D_\flat$ on the subset $(B\times V) \cup (V\times B)$, the desired identity now follows from \ref{li:leibniz}.
\end{proof}

To apply Lemma~\ref{l:fussy-extension} in the case of $\FA(G)$, we need to choose an appropriate dense subspace on which the desired bounds can be verified. We shall do this by considering certain coefficient functions of irreducible representations.

\begin{subsection}{Coefficient functions for the $ax+b$ group}
The irreducible unitary representations of $G$ can be found by identifying it with $\Real\rtimes\Real_+^*$ and using the Mackey machine for induced representations. It follows from this method that (up to unitary equivalence) there are exactly two infinite-dimensional irreducible representations in $\widehat{G}$, which we denote by $\pi_+$ and $\pi_{-}$\/.
These have various different realizations, and readers should beware that different standard sources for non-abelian harmonic analysis often differ in their choices. We follow the description used in \cite{eymard-terp,khalil_ax+b}.

We realize both $\pi_+$ and $\pi_{-}$ as representations on the  Hilbert space $\cH= L^2(\Real^*_+,t^{-1}dt)$, as follows:
\begin{equation}
\pi_{\pm}(b,a)\xi(t) \defeq e^{\mp 2\pi i bt}\xi(at).
\end{equation}
Then the coefficient functions of $\pi_+$ and $\pi_-$ have the following explicit form:
\begin{equation}\label{eq:explicit-coefficients}
\begin{aligned}
(\xi *_{\pi_+} \eta )(b,a)
& = \int_0^\infty e^{-2\pi i bt} \xi(at)\overline{\eta(t)}\, t^{-1}dt, \\
(\xi *_{\pi_-} \eta )(b,a)
& = \int_0^\infty e^{2\pi i bt} \xi(at)\overline{\eta(t)}\, t^{-1}dt & \quad(\xi,\eta\in\cH).
\end{aligned} 
\end{equation}

Suppressing mention of $G$, we denote by $\FA_{\pi_+}$ and $\FA_{\pi_{-}}$ the closed subspaces of $\FA(G)$ generated by the coefficient functions of $\pi_+$ and $\pi_{-}$ respectively.
Note that for $\xi$ and $\eta$ in $\cH$, we have $\overline{\xi*_{\pi_+}\eta}=\overline{\xi}*_{\pi_-}\overline{\eta}$.
Thus $\overline{\FA_{\pi_+}}=\FA_{\pi_-}$.

\begin{prop}[Decomposition of $\FA(G)$]\label{p:A(ax+b)-decomp}
There is an $\ell^1$-direct sum decomposition $\FA(G)=\FA_{\pi_+}\oplus_1 \FA_{\pi_-}$.
\end{prop}

A proof of this, using the Plancherel formula for $G$, is given in \cite[Th\'eor\`eme 4]{khalil_ax+b}. Since the Plancherel theorem itself is not entirely straightforward for this group, and is somewhat technical to prove, we provide an alternative argument in the appendix.

\medskip
We continue with preliminaries.
For our later calculations, it is useful to express coefficient functions of $\pi_+$ and $\pi_{-}$ in terms of the classical Fourier transform on $\Real$. To avoid ambiguity we pause to fix some notational conventions.

\begin{notn}[Normalization for the Fourier transform on $\Real$]
Elements of $\widehat{\Real}$ will be denoted by $\chi_b$, where $b\in\Real$ and $\chi_b(x) = e^{2\pi i bx}$, and we normalize Haar measure on $\Real$ so that the correspondence $b\leftrightarrow \chi_b$ is measure-preserving.
We then define the Fourier transform $\cF:L^2(\Real)\to L^2(\widehat{\Real})$ to be the unique \emph{unitary} map that satisfies
\[ \cF(f)(\chi_b) = \int_{\Real} f(x)\overline{\chi_b(x)} \,dx \qquad\text{for all $f\in (L^1\cap L^2)(\Real), b\in\Real$.} \]
(Our choice of normalization in the definition of $\cF$ follows the choices in \cite[Chapter~4]{Foll_AHAbook}. See the remarks following \cite[Coroll.~4.2.3]{Foll_AHAbook} for further explanations.)
\end{notn}

Given $\xi,\eta\in\cH$ and $a\in \Real^*_+$, define ${}_a\xi\in\cH$ by ${}_a\xi(y)=\xi(ay)$. The map $\xi \mapsto {}_a\xi$ is an isometry on $\cH$, so by Cauchy--Schwarz, $_a\xi \overline{\eta} \in L^1(\Real_+^*, t^{-1}dt)$. 
Let $K^{-1}$ be the linear operator defined formally by
\[ (K^{-1}f)(t) =t^{-1}f(t) \qquad(t\in\Real_+^*) \]
 then $K^{-1}(_a\xi\overline{\eta}) \in L^1(\Real^*_+, dt)$.
Now let $\imath: L^1(\Real^*_+)\to L^1(\Real)$ be the inclusion map, and regard $\cF$ as a map $L^1(\Real)\to C_0(\widehat{\Real})$. Then by Equation \ref{eq:explicit-coefficients}, we have
\begin{equation}\label{eq:coeff-as-fourier}
(\xi*_{\pi_+}\eta)(b,a)=\cF(\imath K^{-1}({}_a\xi\overline{\eta}))(\chi_b)
\quad \mbox{and} \quad
(\xi*_{\pi_-}\eta)(b,a)=\cF(\imath K^{-1}(_a\xi\overline{\eta}))(\chi_{-b}).
\end{equation}

\begin{dfn}
We say that a coefficient function in $\FA_{\pi_+}$ or $\FA_{\pi_-}$ is \dt{convenient} if it is
of the form
 $f=\xi*_{\pi_{\pm}}\eta$ for some $\xi,\eta\in C_c^2(\Real_+^*)$. Note that this definition only applies to coefficient functions and not to their linear combinations.
Define $\cC_+$ to be the subspace of $\FA_{\pi_+}$ spanned by convenient coefficient functions associated to $\pi_+$, define $\cC_{-}$ analogously, and put $\cC\defeq \cC_+ + \cC_{-}\subset\FA(G)$.
\end{dfn}

Since $C_c^2(\Real_+^*)$ is dense in $\cH$, every coefficient function in $\FA_{\pi_{\pm}}$ can be approximated in norm by a convenient coefficient function in $\cC_{\pm}$. Hence $\cC_{\pm}$ is a dense subspace of $\FA_{\pi_{\pm}}$, and since $\FA(G)=\FA_{\pi_+} \oplus_1 \FA_{\pi_-}$, it follows that $\cC$ is a dense subspace of $\FA(G)$.

The formulas \eqref{eq:explicit-coefficients} and the dominated convergence theorem show that each $f\in\cC$ is differentiable in the $b$-direction. Moreover, if $\xi,\eta\in C_c^2(\Real_+^*)$, a direct calculation yields the identity
\[  \begin{aligned}
- \frac{1}{2\pi i} a \frac{\partial}{\partial b}(\xi*_{\pi_+}\eta)(b,a)
& = \int_0^\infty \left(- \frac{1}{2\pi i} \frac{\partial}{\partial b} e^{- 2\pi ibt}\right) a\xi(at)\overline{\eta(t)}\,t^{-1} dt \\
& = \int_0^\infty e^{- 2\pi ibt}  at\ \xi(at)\overline{\eta(t)}\,t^{-1} dt \\
& =  (K\xi *_{\pi_+} \eta) (b,a)\;,
\end{aligned} \]
and similarly (or by conjugation) we have
\[
- \frac{1}{2\pi i} a \frac{\partial}{\partial b}(\xi*_{\pi_-}\eta)(b,a)
 = - (K\xi *_{\pi_{-}} \eta) (b,a)\;.
\]

 To deal with sums of coefficient functions, it is useful to introduce the bounded \emph{linear} maps
\begin{subequations}
\begin{equation}\label{eq:Psi}
\Psi: \cH\ptp\overline{\cH}\to \FA_{\pi_+}
\quad,\quad
 \Psi(\xi\otimes \overline{\eta})  =\xi*_{\pi_+}\eta
\end{equation}
and
\begin{equation}\label{eq:Psi-bar}
\Psibar:  \overline{\cH}\ptp\cH \to \overline{\FA_{\pi_+}}=\FA_{\pi_-}
\quad,\quad
 \Psibar(\overline{\xi}\otimes \eta)  =\overline{\xi}*_{\pi_-}\overline{\eta} = \overline{\Psi(\xi\otimes\overline{\eta})}
\end{equation}
\end{subequations}
Note that $\Psi$ is a surjective quotient map, by definition. Since $\pi_+$ is irreducible, by Remark~\ref{r:A_pi_of_irred} $\Psi$ is injective, and hence a surjective isometry. The same is true for $\Psibar$. We may now rewrite the earlier identities more concisely as
\begin{equation}\label{eq:deriv-coeff-formula}
\left.\begin{aligned}
\madb \Psi(\xi\tp \overline{\eta}) & = \Psi(K\xi\tp \overline{\eta} ) \\
\madb \Psibar (\overline{\xi}\tp \eta) & = -\Psibar(K\overline{\xi}\tp \eta )
\end{aligned}\right\}
\qquad\text{for all $\xi,\eta\in C_c^2 (\Real_+^*)$.}
\end{equation}
(That is, $\Psi$ and $\Psibar$ intertwine the densely-defined operator $\madb$ with the densely-defined operators $K\tp I_{\cH}$ and $-K\tp I_{\cH}$ respectively.)

\begin{lem}[Convenient functions behave well]\label{l:containment1}
$\madb(\cC)\subseteq \cC\subseteq (\FA \cap L^1)(G)$.
\end{lem}

\begin{proof}
By \eqref{eq:deriv-coeff-formula}, since $K(C_c^2(\Real_+^*)) = C_c^2(\Real_+^*)$,
 $\madb$ takes convenient coefficient functions to convenient coefficient functions, and hence $\madb(\cC)\subseteq \cC$ by linearity. This gives the first inclusion.
For the second inclusion, it suffices to prove that convenient coefficient functions are integrable; and since $\cC_{-}=\overline{\cC_+}$, it is enough to consider the convenient coefficient functions associated to $\pi_+$.

 So, let $\xi,\eta\in C_c^2(\Real_+)$ and put $f=\xi *_{\pi_+} \eta$.
Recall from \eqref{eq:coeff-as-fourier} that
\[ f(b,a) = \cF (\imath K^{-1}(_a\xi\overline{\eta}))(\chi_b) \qquad(a\in\Real_+^*,b\in\Real). \]

Since $\supp(\xi)$ and $\supp(\eta)$ are compact subsets of $(0,\infty)$ and $(_a\xi\overline{\eta}(t))=\xi(at)\overline{\eta(t)}$, we see that $_a\xi\overline{\eta}=0$ for all $a$ outside some compact subset $S\subset(0,\infty)$. Hence $\supp(f)\subseteq \Real\times S$.
It~now suffices to show that $\sup_{a\in S} \norm{f(\cdot, a)}_{L^1(\Real)} < \infty$.

Observe that
\[
\imath K^{-1}(_a\xi\overline{\eta}) = {}_a\imath(\xi) \cdot \imath(K^{-1}\overline{\eta})
\]
and that both $\imath(\xi)$ and $\imath(K^{-1}\overline{\eta})$ belong to $C_c^2(\Real)$. We now use the following properties of the Fourier transform on $\Real$:
\begin{newnum}
\item if $h\in C_c^2(\Real)$ then $h$ and $\cF(h)$ are integrable;
\item if $g_1$, $g_2$, $\cF(g_1)$ and $\cF(g_2)$ are all integrable, then $g_1g_2$ is integrable and $\cF(g_1g_2) = \cF(g_1)*\cF(g_2)$;
\item if $g$ and $\cF(g)$ are integrable, and $_ag(t)\defeq g(at)$, then $\norm{\cF(_ag)}_{L^1(\widehat{\Real})}=\norm{\cF(g)}_{L^1(\widehat{\Real})}$.
\end{newnum}
(The last property can be verified by direct calculation, but it is also a special case of the general fact that continuous automorphisms of a locally compact group, in this case $\Real$, induce \emph{isometric} automorphisms of its Fourier algebra.)
Together, these properties imply that
\[ \begin{aligned}
\norm{f(\cdot, a)}_{L^1({\Real})}
& = \norm{\cF(_a\imath(\xi)) *
 \cF(\imath(K^{-1}\overline{\eta}))}_{L^1(\widehat{\Real})} \\
& \leq \norm{\cF(_a\imath(\xi)) }_{L^1(\widehat{\Real})}\
  \norm{\cF(\imath(K^{-1}\overline{\eta}))}_{L^1(\widehat{\Real})} 
& = \norm{\cF(\imath(\xi)) }_{L^1(\widehat{\Real})}\
  \norm{\cF(\imath(K^{-1}\overline{\eta}))}_{L^1(\widehat{\Real})} \;.
\end{aligned} \]
This gives us the required uniform bound on $L^1$-norms, and hence concludes the proof.
\end{proof}

\begin{cor}[A convenient algebra]\label{c:containment2}
Let $\cB$ be the not-necessarily closed subalgebra of $\FA(G)$ generated by the set of convenient coefficient functions. Then $\madb(\cB) \subseteq (\FA\cap L^1)(G)$.
\end{cor}

\begin{proof}
By linearity it is enough to prove that $\madb(f)\in (\FA\cap L^1)(G)$ whenever $f$ is a product of finitely many convenient coefficient functions. This follows by induction, using the product rule and the inclusions from Lemma~\ref{l:containment1}.
\end{proof}

\end{subsection}

\begin{subsection}{Orthogonality relations and estimates for our derivation}
\label{ss:ax+b_orth}
Consider the following densely defined, symmetric operator on $\cH$:
\[ (K\xi)(t)=t\xi(t) \qquad (t\in\Real_+^*) .\]
Note that $K^{-1}$ is also densely defined and symmetric, and is given by $(K^{-1}\xi)(t)=t^{-1} \xi(t)$. We have already made use of $K^{-1}$ in other calculations.

\begin{rem}
$K$ is the so-called \dt{Duflo--Moore operator} for the representations $\pi_+$ and $\pi_{-}$. It is a special case of a more general construction due to Duflo and Moore in \cite{duflo-moore} for certain non-unimodular groups.
However, to keep our arguments self-contained, we will not rely on the results of~\cite{duflo-moore}. The $ax+b$ group is sufficiently simple that it would take more effort to precisely translate those results into our setting, than to just carry out the necessary calculations directly.
\end{rem}

The following identities are special cases of known results. Since various treatments in the literature of the $ax+b$ group adopt different conventions/normalizations, and in some cases work with different (but unitarily equivalent) representations, we give a full statement and proof of these identities for sake of completeness.

\begin{prop}[Explicit orthogonality relations]\label{p:orthogonality-ax+b}
Let $\eta_1,\eta_2,\xi_1,\xi_2\in C_c^2(\Real_+^*)$. Then
%
\begin{subequations}
\begin{align}
\label{eq:ax+b_orth++}
\pair{  \xi_1*_{\pi_+}\eta_1 }{ \xi_2*_{\pi_+}\eta_2 }_{L^2(G)}
 & =  \pair{ \eta_2 }{ \eta_1 }_{\cH}\pair{  K^{-\frac{1}{2}}\xi_1 }{ K^{-\frac{1}{2}}\xi_2 }_{\cH}.\\
\label{eq:ax+b_orth--}
 \pair{  \xi_1*_{\pi_-}\eta_1 }{ \xi_2*_{\pi_-}\eta_2 }_{L^2(G)}
 & =  \pair{ \eta_2 }{ \eta_1 }_{\cH}\pair{  K^{-\frac{1}{2}}\xi_1 }{ K^{-\frac{1}{2}}\xi_2 }_{\cH}.\\
\label{eq:ax+b_orth+-}
\pair{  \xi_1*_{\pi_+}\eta_1 }{ \xi_2*_{\pi_-}\eta_2 }_{L^2(G)} & =  0.
\end{align}
\end{subequations}
\end{prop}

\begin{proof}
We already know that convenient coefficient functions belong to $(C_0\cap L^1)(G)$, so certainly they belong to $L^2(G)$; thus the inner products on the left-hand sides of  \eqref{eq:ax+b_orth++}, \eqref{eq:ax+b_orth--} and \eqref{eq:ax+b_orth+-} are all well-defined.

We treat \eqref{eq:ax+b_orth++} first. By \eqref{eq:coeff-as-fourier}, and unitarity of $\cF$,
\[ \begin{aligned}
\pair{  \xi_1*_{\pi_+}\eta_1 }{ \xi_2*_{\pi_+}\eta_2 }_{L^2(G)}
& =
 \int_{\Real_+}\int_\Real \xi_1*_{\pi_+}\eta_1(b,a)\overline{\xi_2*_{\pi_+}\eta_2(b,a)}
\ db\ \frac{da}{a^2}\\
& = 
 \int_{\Real_+} \pair{ \cF\imath(K^{-1}(_a\xi_1\overline{\eta_1})) }{ \cF\imath(K^{-1}(_a\xi_2\overline{\eta_2}) )}_{L^2(\widehat{\Real})} \; \frac{da}{a^2}\\
& =
 \int_{\Real_+} \pair{ \imath K^{-1}(_a\xi_1\overline{\eta_1}) }{ \imath K^{-1}(_a\xi_2\overline{\eta_2}) }_{L^2(\Real)} \; \frac{da}{a^2}\,.
\end{aligned} \]
But now, direct calculation shows this is equal to
\[ \begin{aligned}
 & \phantom{=} 
 \int_{\Real_+} \int_{\Real_+}\frac{\xi_1(ab)\overline{\eta_1(b)}}{b}
  \frac{\overline{\xi_2(ab)} \eta_2(b)}{b}\,db\;\frac{da}{a^2}\\
& =
 \int_{\Real_+}\int_{\Real_+}\frac{\xi_1(ab)\overline{\xi_2(ab)}}{ab^2}\overline{\eta_1(b)}{\eta_2(b)}\frac{da}{a}db\\
& =
 \int_{\Real_+}\int_{\Real_+}\frac{\xi_1(a)\overline{\xi_2(a)}}{ab}\overline{\eta_1(b)}{\eta_2(b)}\frac{da}{a}db\\
& =
 \pair{ \eta_2 }{ \eta_1 }_{\cH}\pair{  K^{-\frac{1}{2}}(\xi_1) }{ K^{-\frac{1}{2}}(\xi_2) }_{\cH},
\end{aligned} \]
where we used the change of variable $a\mapsto \frac{a}{b}$, and the fact that $a^{-1}da$ is invariant under multiplication. This shows that \eqref{eq:ax+b_orth++} holds. The proof that \eqref{eq:ax+b_orth--} holds is similar and we omit the details.

To prove that \eqref{eq:ax+b_orth+-} holds, consider the operators $V_{\pm}:\cH\rightarrow L^2(G)$, $V_{\pm}(\eta)(b,a)=\pair{ \eta }{ \pi_\pm(b,a)\xi_\pm }$, where
$\xi_+= {\norm{K^{-\frac{1}{2}}\xi_1}_{\cH}}^{-1}\xi_1$
and
$\xi_-= {\norm{K^{-\frac{1}{2}}\xi_2}_{\cH}}^{-1}\xi_2$. 
It follows from the first two orthogonality relations that $V_+$ and $V_-$ are isometries, intertwining $\lambda$ with $\pi_+$ and $\pi_{-}$ respectively. Thus $V_-^*V_+$ intertwines $\pi_+$ and $\pi_-$, so equals $0$ by Schur's lemma. Hence
 \[ \pair{  \xi_1*_{\pi_+}\eta_1 }{ \xi_2*_{\pi_-}\eta_2 }_{L^2(G)}=\pair{  V_{\xi_2}(\eta_2) }{ V_{\xi_1}(\eta_1) }_{L^2(G)}
 =0 \] 
as required.
\end{proof}

\begin{rem}
The orthogonality relations hold in greater generality: namely, whenever $\eta_1,\eta_2\in \cH$ and $\xi_1,\xi_2\in \dom(K^{-\frac{1}{2}})$. Since we only need orthogonality for convenient coefficient functions, we omit the details.
\end{rem}

\begin{prop}[The key estimate]\label{p:ax+b_mainbound}
Define a bilinear map $D_\flat: \cC \times \cC \to \Cplx $ by
\begin{equation}
D_\flat(f,g) = \int_G (\madb f) g  \,d\mu
 = -\frac{1}{2\pi i}\int_G a\frac{\partial}{\partial b} f(b,a) g(b,a) \,d\mu(b,a) \ .
\end{equation}
Then 
$\abs{ D_\flat(f,g)} \leq \norm{f}_{\FA(G)} \norm{g}_{\FA(G)}$
for all $f,g\in \cC$.
\end{prop}

\begin{proof}
Let $f,g\in \cC$. We may write $f=f_1+f_2$ and $g_1=g_1+g_2$ where $f_1,g_1\in \cC_+$ and $f_2,g_2\in \cC_{-}$; moreover, $\norm{f}_{\FA(G)}=\norm{f_1}_{\FA(G)}+\norm{f_2}_{\FA(G)}$ and
$\norm{g}_{\FA(G)}=\norm{g_1}_{\FA(G)}+\norm{g_2}_{\FA(G)}$. 
By~\eqref{eq:deriv-coeff-formula}, the orthogonality relations in Proposition~\ref{p:orthogonality-ax+b}, and the fact
that $\overline{A_{\pi_+}}=A_{\pi_{-}}$,
\[ \begin{aligned} 
D_\flat(f_1,g_1)  = \pair{ \madb(f_1) }{\overline{g_1} }_{L^2(G)} = 0\,, \\
D_\flat(f_2,g_2)  = \pair{ \madb(f_2) }{\overline{g_2} }_{L^2(G)} = 0\,.
\end{aligned} \]

Define the contractive linear map 
 \begin{equation*}
 \Phi:\cH\ptp\overline{\cH}\ptp\overline\cH\ptp{\cH} \rightarrow \Cplx, \quad 
 \Phi(\xi_1\otimes\overline{\eta_1}\otimes\overline{\xi_2}\otimes{\eta_2})=\pair{ \xi_1 }{ \xi_2 }\pair{ \eta_2 }{ \eta_1 },
 \end{equation*}
and recall that $\Psi\otimes\Psibar: \cH\ptp\overline{\cH} \ptp\overline{\cH}\ptp \cH \to \FA_{\pi_+}\ptp \FA_{\pi_-}$ is an isometric isomorphism.

We claim that
\[ D_\flat(f_1, g_2)= (\Phi\circ(\Psi\otimes\Psibar)^{-1})(f_1\otimes g_2). \]
For, by linearity, it suffices to verify this in the special case when $f_1$ and $g_2$ are convenient coefficient functions (associated to $\pi_+$ and $\pi_{-}$ respectively). So, suppose that
\[ f_1 = \Psi(\xi_1\otimes\overline{\eta_1})\quad,\quad g_2 = \Psibar(\overline{\xi_2}\otimes\eta_2) = \overline{\Psi(\xi_2\tp\overline{\eta_2})} \]
for some $\xi_1,\xi_2,\eta_1,\eta_2\in  C_c^2(\Real_+^*)$. Then 
\eqref{eq:deriv-coeff-formula} and \eqref{eq:ax+b_orth++} give
\[ \begin{aligned}
D_\flat( f_1 , g_2 )
 & = \int_G (\madb f_1) g_2  \,d\mu \\
 & =  \pair{ \Psi(K\xi_1 \tp \overline{\eta_1})  }{ \Psi(\xi_2\tp \overline{\eta_2}) }_{L^2(G)}\\
 & =   \pair{K^{-\frac{1}{2}}K\xi_ 1}{ K^{-\frac{1}{2}}\xi_2}_\cH
	\pair{\eta_2}{\eta_1}_\cH \\
 & =  \pair{ \xi_1 }{ \xi_2 }_\cH\pair{ \eta_2 }{ \eta_1 }_\cH
 & =  \Phi\circ(\Psi\otimes \Psibar)^{-1}(f_1 \otimes g_2 ),
\end{aligned} \]
as required. This proves our claim.

Arguing similarly, we also have
\[ D_\flat(f_2,g_1) = \Phi\circ(\Psibar\otimes\Psi)^{-1}( f_2 \tp g_1). \]
Hence, putting everything together,
\[ \begin{aligned}
|D_\flat(f_1+f_2,g_1+g_2)|
&= |D_\flat(f_1,g_2)+D_\flat(f_2,g_1)|  \\
&\leq  |D_\flat(f_1,g_2)|+|D_\flat(f_2,g_1)|
\\
&\leq  \norm{f_1}_{\FA(G)}\norm{g_2}_{\FA(G)}+ \norm{f_2}_{\FA(G)} \norm{g_1}_{\FA(G)}
\\
&\leq  (\norm{f_1}_{\FA(G)}+\norm{f_2}_{\FA(G)})(\norm{g_1}_{\FA(G)}+\norm{g_2}_{\FA(G)}) \\
& = \norm{f}_{\FA(G)}\norm{g}_{\FA(G)}
\end{aligned} \]
and the proposition is proved.
\end{proof}

We now have everything in place.
\begin{thm}\label{t:ax+b_non-WA}
There is a continuous extension of $D_\flat$ to a non-zero, bounded, cyclic derivation $D:\FA(G)\to\FA(G)^*$. In particular, $\FA(G)$ is not cyclically amenable, so is not weakly amenable.
\end{thm}

\begin{proof}
Observe that the bilinear map $D_\flat:\cC\times\cC\to\Cplx$ is not identically zero, since\hfill\newline
$D_\flat(\xi*_{\pi_+}\xi,\overline{\xi*_{\pi_+}\xi})=\norm{\xi}^4$ for all $\xi\in C_c^2(\Real_+^*)$.

Let $\cB$ be the algebra generated by the convenient coefficient functions. 
We wish to apply Lemma~\ref{l:fussy-extension} to the bilinear map $D_\flat:\cC\times\cC\to\Cplx$, with $A=\FA(G)$, $V=\cC$ and $B=\cB$.
Recall for sake of clarity that
\[ D_\flat( f,g) = \int_G \madb(f)g\,d\mu = -\frac{1}{2\pi i} \int_0^\infty \int_{-\infty}^\infty a \frac{\partial f}{\partial b}(b,a) g(b,a)\; db\frac{da}{a^2} \qquad(f,g\in\cB). \]
By Corollary~\ref{c:containment2} we have $\madb(\cB)\subseteq (\FA\cap L^1)(G)$. So we may extend $D_\flat$ to a well-defined bilinear map $\cB\times\cB\to\Cplx$, which is $\norm{\cdot}_{\FA(G)}$-continuous in the second variable. (This is the place where we need $\madb$ to take convenient functions to \emph{integrable} ones.)
On the other hand, integrating by parts (which is justified, since $M_a\partial_b(\cB)\subseteq C_0(G)$) we see that $D_\flat$ is an \emph{anti-symmetric} bilinear map on $\cB\times\cB$; therefore it is also $\norm{\cdot}_{\FA(G)}$-continuous in the \emph{first} variable.

Thus conditions~\ref{li:bounded1} and \ref{li:leibniz} of Lemma~\ref{l:fussy-extension} hold. We already know by Proposition~\ref{p:ax+b_mainbound} that condition~\ref{li:bounded2} of that lemma holds. Therefore, $D$ is a non-zero derivation $\FA(G)\to\FA(G)^*$. It is a cyclic derivation, since $D_\flat$ is an antisymmetric bilinear form.
\end{proof}

\end{subsection}

\end{section}

\begin{section}{Application to Fourier algebras on connected Lie groups}\label{s:lie-struct}
By Proposition~\ref{p:WA-hered} and the results of the previous sections, if we wish to show $\FA(G)$ is not cyclically amenable, it suffices to show that $G$ contains a closed copy of any of the following groups: $\SU_2(\Cplx)$, $\SO_3(\Real)$, or the real $ax+b$ group.
In this section we will use structure theory for Lie groups to show that many connected Lie groups have closed copies of these key examples, and hence have Fourier algebras which are not cyclically amenable.

Let us review some definitions and fix some terminology; our sources are \cite{HN_LieBook} for general background results in Lie theory, and
\cite{Knapp_beyond} for the Iwasawa decomposition. To be consistent with these sources, we follow the convention that ``\dt{simply connected}'' means ``trivial $\pi_1$\/'', and reserve the term ``\dt{$1$-connected}'' to mean ``trivial~$\pi_0$ and~$\pi_1$''. Thus a Lie group is $1$-connected if and only if it is both connected and simply connected.

\begin{dfn}[Semisimple Lie groups]
Let $G$ be a Lie group. The \dt{(solvable) radical} of $G$, denoted by $\rad(G)$, is the largest solvable, connected, normal subgroup of~$G$. If $\rad(G)=\{e_G\}$ then we say $G$ is \dt{semisimple.}
\end{dfn}

The following standard result from the theory of compact Lie groups has been mentioned earlier (Remark~\ref{r:copy-of-so3-or-su2}) but we restate it for emphasis.
\begin{prop}\label{p:compact-case}
Every compact, connected, non-abelian Lie group contains either a closed copy of $\SO_3(\Real)$ or $\SU_2(\Cplx)$.
\end{prop}

The key point is that the Lie algebra of such a group is a semisimple real Lie algebra, and basic structure theory for such algebras implies the existence of a subalgebra isomorphic to~$\fsl_2(\Real)$.
One then exponentiates this subalgebra and appeals to further results from Lie theory to show that the subgroup generated in this way is closed.

The next result seems equally well-known to specialists, but we were unable to locate an explicit statement in the literature.

\begin{prop}\label{p:ss-non-compact}
Every non-compact, connected, semisimple Lie group contains a closed subgroup isomorphic to the $ax+b$ group.
\end{prop} 

Our proof relies on the Iwasawa decomposition of such a group: see, for instance, \cite[\S VI.4]{Knapp_beyond}.
It incorporates some suggestions communicated to the first author by V.~Protsak on the {\it MathOverflow} website.

\begin{proof}
Let $G$ be such a group and let $\fg$ be its Lie algebra. The Iwasawa decomposition of $\fg$ exhibits it as a direct sum of subalgebras $\fg=\fk\oplus\fa\oplus\fn$, where $\fa$ is abelian and $\fn$ is nilpotent. Since $G$ is not compact, the construction of the Iwasawa decomposition ensures we also have the following properties:
\begin{newnum}
\item both $\fa$ and $\fn$ are non-zero;
\item for each $x\in \fa$, the operator $\ad_x:\fg\to\fg$ is diagonalizable with real eigenvalues, and maps $\fn$ to $\fn$.
\end{newnum}
(See \cite[Proposition 6.43 and Lemma 6.45]{Knapp_beyond} for details.)
In particular, there exists $u\in\fa$ such that $\ad_u: \fn\to\fn$ has eigenvalue~$1$. Let $v$ be any corresponding eigenvector, and let $\fh=\Real u +\Real v$. This is a direct sum, since $u$ and $v$ are linearly independent. Note that, up to isomorphism of Lie groups, the $ax+b$ group is the unique $1$-connected Lie group with Lie algebra~$\fh$.

Let $AN$ be the subgroup of $G$ obtained by applying the exponential map $\exp_G: \fg \to G$ to the subalgebra $\fa\oplus\fn$. By the proof or construction of the Iwasawa decomposition (see e.g.~\cite[Theorem 6.46]{Knapp_beyond}), $AN$ is a closed and $1$-connected subgroup of~$G$. Therefore, any subgroup obtained by applying $\exp_G$ to a subalgebra of $\fa\oplus\fn$ is also closed and $1$-connected. (This follows from \cite[Lemma 6.44]{Knapp_beyond}; it is also a special case of more general results about $1$-connected solvable Lie groups, see \cite[Proposition 11.2.15]{HN_LieBook}.)

So putting $H=\exp_G(\fh)$ we see that $H$ is a closed, $1$-connected subgroup of $G$, whose Lie algebra is $\fh$. By the remarks of the previous paragraph, $H$ is the desired copy of the $ax+b$ group.
\end{proof}

What about connected Lie groups which are not semisimple? Here matters become more complicated if the fundamental group is non-trivial, since subalgebras of the Lie algebra of a group do not in general exponentiate to give closed subgroups. We therefore restrict attention to the $1$-connected setting, where the passage between Lie algebras and Lie groups works best.

\begin{prop}
Let $G$ be a $1$-connected Lie group that is not solvable. Then $G$ contains a closed subgroup isomorphic to $\SU_2(\Cplx)$, $\SO_3(\Real)$ or the $ax+b$ group.
\end{prop}

\begin{proof}
Let $\fg$ and $\fr$ be the Lie algebras of $G$ and $\rad(G)$ respectively. Then $\fr$ is the largest solvable ideal in $\fg$, and it is a proper ideal since $G$ is not solvable. The quotient Lie algebra $\fs\defeq\fg/\fr$ is a semisimple Lie algebra, i.e.~it has no non-zero solvable ideals. By Levi's theorem (e.g.~\cite[Theorem 5.6.6]{HN_LieBook}) there is a Lie subalgebra $\fs\subseteq\fg$ such that $\fg\iso\fr\rtimes \fs$. Since $G$ is simply connected, this decomposition can be exponentiated to an isomorphism of Lie groups $G\iso\rad(G)\rtimes S$, where $S$ is the unique $1$-connected Lie group with Lie algebra~$\fs$. (See e.g.~\cite[\S1.12]{Knapp_beyond}.) Since $\fs$ is a non-zero semisimple Lie algebra, $S$ is a non-trivial, connected semisimple Lie group. The rest now follows from Propositions~\ref{p:compact-case} and \ref{p:ss-non-compact}.
\end{proof}

From this and Proposition~\ref{p:WA-hered}, we obtain the final theorem of this section.
\begin{thm}\label{t:should-be-solvable}
Let $G$ be a $1$-connected Lie group such that $\FA(G)$ is cyclically amenable. Then $G$ is solvable, and contains no closed copy of the $ax+b$ group.
\end{thm}

It is natural to wonder what can be said for connected nilpotent Lie groups. The most obvious example is the real Heisenberg group, which lies outside the reach of the present paper. However, our techniques can be applied to a certain natural quotient of the Heisenberg group, and this will be the subject of the next section.
\end{section}


\section{The reduced Heisenberg group}\label{s:reduced-H}
We define the \dt{reduced Heisenberg group} $\bbH_r$ to be the set $\{(p,q,e^{2\pi i \theta}):p,q\in \Real,\theta\in[0,1)\}$, with group multiplication defined as 
\begin{equation}
(p,q,e^{2\pi i \theta})\cdot(p',q',e^{2\pi i \theta'})=(p+p',q+q',e^{2\pi i (\theta+\theta')}e^{\pi i (pq'-qp')})
\end{equation}
Note for later reference that $(p,q,e^{2\pi i\theta})^{-1}= (-p,-q,e^{-2\pi i\theta})$.
$\bbH_r$ is a $2$-step nilpotent Lie group, with centre $Z(\bbH_r) = \{(0,0,z)\colon z\in\Torus\}$.
Haar measure is easily described: give $\Torus$ the normalized Haar measure $d\theta$, and $\Real$ usual Lebesgue measure $dp$. 
Then the left Haar measure $\mu$ of $\bbH_r$ is defined by $d\mu(p,q,e^{2\pi i \theta})=dp\,dq\,d\theta$, which is a right Haar measure as well.

\begin{dfn}
Let $n\in \Zahl\setminus\{0\}$. The \dt{Schr\"odinger representation} $\sch_n:\bbH_r\rightarrow \cU(L^2(\Real))$ is defined by
\[ \sch_n(p,q,e^{2\pi i \theta})\xi(x)=e^{2\pi i nq(-x+\frac{p}{2})}e^{2\pi i n\theta}\xi(-p+x).\]
\end{dfn}

It is known that every $\sch_n$ is irreducible; this follows from an application of the Mackey machine if we identify $\bbH_r$ with a semidirect product $(\Real\times\Torus)\rtimes\Real$ in a suitable way. We will also see shortly (Proposition~\ref{p:red-heisenberg-sq-integ}) that $\sch_n$ is also square-integrable.

Since $\bbH_r$ is not an AR-group, we cannot decompose $\FA(\bbH_r)$ into a direct sum of coefficient spaces of irreducible representations. However, it turns out that a sufficiently large part of $\FA(\bbH_r)$ can be decomposed in this way: namely, the part generated by coefficient functions associated to the Schr\"odinger representations. This will be enough for us to carry out the same approach as in Section~\ref{ss:ax+b_orth}.

\subsection{Coefficient functions associated to the Schr\"odinger representations}
Let $n\in\Zahl\setminus\{0\}$. Given $\xi,\eta\in L^2(\Real)$, we note for later reference that
$\overline{\xi*_{\sch_n}\eta} = \overline{\xi}*_{\sch_{-n}}\overline{\eta}$.
We also have the following useful identity:
\begin{equation}\label{eq:sch-coeff-using-FT}
\begin{aligned}
(\xi *_{\sch_n} \eta )(p,q,e^{2\pi i \theta})
& = \int_{\Real} e^{2\pi i nq(-x+\frac{p}{2})}e^{2\pi i n \theta}\xi(-p+x)\overline{\eta(x)}dx \\
& = e^{2\pi i n\theta}\cF(_{-\frac{p}{2}}\xi\ _{\frac{p}{2}}\overline{\eta})(nq), 
\end{aligned}
\end{equation}
where $(_a\xi)(t)\defeq \xi(t-a)$, and $\cF$ denotes the Fourier transform (with the same normalization used in Section~\ref{s:ax+b}).

\begin{prop}\label{p:red-heisenberg-sq-integ}
Let $n$ be a nonzero integer, and let $\xi,\eta\in C_c^\infty(\Real)$. Then $\xi *_{\sch_n} \eta \in L^2(\bbH_r)$, with
\begin{equation}\label{eq:two-norm-of-coeff}
\norm{ \xi *_{\sch_n} \eta }_2^2 =\frac{1}{|n|}\norm{\xi}_2^2\norm{\eta}_2^2\,.
\end{equation}
In particular, the representation $\sch_n$ is square-integrable.
\end{prop}
\begin{proof}
Fix $p\in\Real$. Since $\cF:L^2(\Real)\to L^2(\Real)$ is a unitary isomorphism,
%
\[ \begin{aligned}
\int_{\Real}|\cF(_{-\frac{p}{2}}\xi\ _{\frac{p}{2}}\overline{\eta})(nq)|^2 \,dq
&= \frac{1}{|n|}\int_{\Real}|\cF(_{-\frac{p}{2}}\xi\ _{\frac{p}{2}}\overline{\eta})(q)|^2 \,dq\\
&= \frac{1}{|n|}\int_{\Real}|\xi(q-\frac{p}{2})|^2|{\eta}(q+\frac{p}{2})|^2\, dq\\
&= \frac{1}{|n|}\int_{\Real}|\xi(q)|^2|{\eta}(q+p)|^2 \, dq<\infty.\\
\end{aligned} \]
Hence
\[ \begin{aligned}
\int_\Real\int_\Real\int_{[0,1)} |e^{2\pi i n\theta}\cF(_{-\frac{p}{2}}\xi\ _{\frac{p}{2}}\overline{\eta})(nq)|^2 \, d\theta\, dq\, dp\,
& = 
\frac{1}{|n|}\int_\Real\int_\Real\int_{[0,1)}|\xi(q)|^2|{\eta}(q+p)|^2  \, d\theta\, dq\, dp \\
& = \frac{1}{|n|}\int_\Real\int_\Real|\xi(q)|^2|{\eta}(q+p)|^2 \, dq\, dp \\
& =\frac{1}{|n|}\norm{\xi}_2^2\norm{\eta}_2^2 \,.
\end{aligned} \]
Putting these together yields Equation~\eqref{eq:two-norm-of-coeff} as required.
\end{proof}

\begin{prop}[Explicit orthogonality relations]\label{p:orthog-rel}
Let  $\xi_1,\xi_2,\eta_1,\eta_2\in \ccinf(\Real)$, and $m,n\in \Zahl\setminus\{0\}$. Then 
\begin{eqnarray*}
\pair{  \xi_1*_{\pi_n}\eta_1 }{ \xi_2*_{\pi_n}\eta_2 }_{L^2(\bbH_r)}
 & = & \frac{1}{|n|}\pair{ \eta_2 }{ \eta_1 }_{L^2(\Real)}\pair{  \xi_1 }{ \xi_2 }_{L^2(\Real)}.\\
\pair{  \xi_1*_{\pi_n}\eta_1 }{ \xi_2*_{\pi_m}\eta_2 }_{L^2(\bbH_r)}
 & = & 0, \ \text{ whenever $n\neq m$}.
\end{eqnarray*}
\end{prop}
\begin{proof}
The first identity follows from Equation~\eqref{eq:two-norm-of-coeff} and polarization.
The second identity is immediate from the explicit formula in Equation~\eqref{eq:sch-coeff-using-FT} and the orthogonality in $L^2(\Torus)$ of the trigonometric monomials.
\end{proof}

Let $\FA_\sch(\bbH_r)$ denote the closed subspace of $\FS(\bbH_r)$ generated by the subspaces $\FA_{\sch_n}(\bbH_r)$ for all $n\in\Zahl\setminus\{0\}$.
It follows from Proposition~\ref{p:red-heisenberg-sq-integ} that $\bigoplus_{0\neq n\in \Zahl}\sch_n$ is a subrepresentation of the left regular representation $\lambda$ of $\bbH_r$, and so $\FA_\sch(\bbH_r)\subseteq \FA(\bbH_r)$.
Moreover, by Corollary 3.13 and Theorem 3.18 of \cite{Arsac} (see also \cite[Proposition 3.5.18]{Zwarich_MSc}), we can write $\FA_\sch(\bbH_r)$ as an $\ell^1$-direct sum $\bigoplus_{0\neq n\in \Zahl}\FA_{\sch_n}(\bbH_r)$.

There is an obvious quotient homomorphism of topological groups $Q: \bbH_r\to \Real^2$, defined by $Q(p,q,e^{2\pi i\theta}) = (p,q)$. Let $\lambda_0 = \lambda_{\Real^2}\circ Q: \bbH_r\to \cU(L^2(\Real^2))$.
A direct calculation shows that if $\xi,\eta\in C_c(\Real^2)$ then
$\xi*_{\lambda_0}\eta\in C_c(\bbH_r)$, and so $\lambda_0$ is square-integrable.
Thus $\FA_{\lambda_0}(\bbH_r)$ is a closed subspace of $\FA(\bbH_r)$.

\begin{prop}\label{p:A(Hr)-decomp}
There is an $\ell^1$-direct sum decomposition
$\FA(\bbH_r)= \FA_{\sch}(\bbH_r) \oplus_1 \FA_{\lambda_0}(\bbH_r)$.
\end{prop}

Although this decomposition appears to be known to specialists, we were unable to locate an explicit statement of it in the literature.
 The original version of this paper included a proof of Proposition~\ref{p:A(Hr)-decomp} similar to that of Proposition~\ref{p:A(ax+b)-decomp}. We would like to thank the referee for suggesting a simpler proof, which can be found in the appendix.

\subsection{Constructing a cyclic derivation on $\FA(\bbH_r)$}
Our approach is very similar to the one we used for the real $ax+b$ group, and once again we  use Lemma \ref{l:fussy-extension}. Let
\[ V = \lin \{\xi*_{\sch_n}\eta:\xi,\eta\in C_c^\infty(\Real),0\neq n\in\Zahl\}
   + {\lin\{ \xi*_{\lambda_0} \eta \colon  \xi,\eta\in C_c(\Real^2) \} }\,, \]
and let $B$ denote the algebra generated by $V$.
 Note that $V$ is dense in $\FA(\bbH_r)$.
For $f\in C^1(\bbH_r)$ we define
\[ \partial_\theta f \defeq \frac{1}{2\pi i} \frac{\partial f}{\partial \theta} \]
Note that $\partial_\theta$ vanishes on all coefficient functions of $\lambda_0$, since they are constant on cosets of the closed subgroup ${\mathcal Z}\defeq \{(0,0,z):\ z\in \Torus\}$. (In fact, by \cite[Proposition 3.25]{eymard64}, every $f$ in $\FA(\bbH_r)$ which is constant on cosets of ${\mathcal Z}$ belongs to $\FA_{\lambda_0}(\bbH_r)$, although we do not need this in what follows.)

Now define $D_\flat: B\times B\rightarrow \Cplx$ by
\begin{equation}\label{eq:H-r-derivation}
D_\flat(f,g)=\int_{\Real}\int_{\Real}\int_{\Torus}\partial_\theta f(x,y,e^{2\pi i \theta})g(x,y,e^{2\pi i \theta})\, dx\, dy\, d\theta.
\end{equation}
The following lemma ensures that $D_\flat$ is well-defined.
\begin{lem}
$\partial_\theta(B)\subseteq B\subseteq (\FA\cap L^1)(\bbH_r)$.
\end{lem}
\begin{proof}
To show the first inclusion, it is enough (by the product rule) to show that $\partial_\theta(V)\subseteq V$. Since $\partial_\theta$ vanishes on $\FA_{\lambda_0}(\bbH_r)$ it suffices to prove that $\partial_\theta(\xi*_{\sch_n}\eta)\in V$ for every nonzero integer $n$ and every $\xi,\eta\in C_c^\infty(\Real)$.
This is straightforward, since by \eqref{eq:sch-coeff-using-FT},
\[
\begin{aligned}
\partial_\theta(\xi*_{\sch_n}\eta)(x,y,e^{2\pi i \theta})
 & = \partial_\theta(e^{2\pi i n\theta}\cF(_{-\frac{p}{2}}\xi\ _{\frac{p}{2}}\overline{\eta})(nq)) \\
 & = n e^{2\pi i n\theta}\cF(_{-\frac{p}{2}}\xi\ _{\frac{p}{2}}\overline{\eta})(nq) 
 & = n(\xi*_{\sch_n}\eta)(x,y,e^{2\pi i\theta})
\end{aligned}
\]
For the second inclusion, observe that $V\subseteq (\FA\cap L^1)(\bbH_r)$. Therefore, since $(\FA\cap L^1)(\bbH_r)$ is a subalgebra of $\FA(\bbH_r)$, it contains the algebra generated by $V$, which is~$B$.
\end{proof}

\begin{thm}\label{t:reduced-H_non-WA}
There is a continuous extension of $D_\flat$ to a non-zero, bounded, cyclic derivation $D:\FA(\bbH_r)\rightarrow \FA(\bbH_r)^*$. In particular, $\FA(\bbH_r)$ is not cyclically amenable, hence is not weakly amenable. 
\end{thm}
\begin{proof}
To reduce notational clutter, we denote the norm on $\FA(\bbH_r)$ by $\norm{\cdot}_{\FA}$.

We will apply Lemma \ref{l:fussy-extension} to $A=\FA(\bbH_r)$, and $V$ and $B$ defined as above. It is easy to see that, because of its definition, $D_\flat$ satisfies the Leibniz identity on $B$.
Also, $D_\flat$ is nonzero, since
\[ D_\flat(\xi*_{\sch_n}\eta\, ,\, \overline{\xi}*_{\sch_{-n}}\overline{\eta})=D_\flat(\xi*_{\sch_n}\eta\, ,\, \overline{\xi*_{\sch_n}\eta})=\norm{\xi}^2\norm{\eta}^2. \]

We may check that condition \ref{li:bounded1} of  Lemma \ref{l:fussy-extension} is satisfied, by an argument very similar to the one used for the $ax+b$ group (Theorem~\ref{t:ax+b_non-WA}).
Integrating by parts, we see that $D_\flat$ is an antisymmetric bilinear map. Moreover, since $\partial_\theta(B)\subseteq L^1(\bbH_r)$, $D_\flat$ is $\norm{\cdot}_{\FA}$-continuous in the second variable. Therefore by antisymmetry it is also $\norm{\cdot}_{\FA}$-continuous in the first variable.

To check condition \ref{li:bounded2}, first let $v=\xi*_{\sch_n}\eta$ and $w=\xi'*_{\sch_m}\eta'$.
Then
\[ \begin{aligned}
D_\flat(v,w)
 =  \int_{\bbH_r} \partial_\theta (v) w \,d\mu 
 & =  \ip{\partial_\theta(\xi*_{\sch_n}\eta) }{ \overline{\xi'*_{\sch_m}\eta'} }_{L^2(\bbH_r)} \\
 & = \ip{ n(\xi*_{\sch_n}\eta) }{ (\overline{\xi'}*_{\sch_{-m}}\overline{\eta'} )}_{L^2(\bbH_r)}.
\end{aligned} \]
Thus, $|D_\flat(\xi*_{\sch_n}\eta\, ,\, {\xi'*_{\sch_m}\eta'})|=0$  if $n\neq -m$, while in the case of $n=-m$ we get
\[\begin{aligned}
|D_\flat(v,w)|
 =|D_\flat(\xi*_{\sch_n}\eta,\xi'*_{\sch_{-n}}\eta')| 
& =|\pair{ \xi }{ \overline{\xi'}  } \pair{ \overline{\eta'} }{ \eta }| \\
& \leq  \norm{\xi}\norm{\xi'}\norm{\eta}\norm{\eta'}
= \norm{v}_{\FA}\norm{w}_{\FA}.
\end{aligned} \]
Thus condition \ref{li:bounded2} holds for coefficient functions of Schr\"odinger representations. By identifying $\FA_{\sch_n}\ptp\FA_{\sch_{-n}}$ with $L^2(\Real)\ptp\overline{L^2(\Real)}\ptp\overline{L^2(\Real)}\ptp L^2(\Real)$,
one can extend this to show that condition \ref{li:bounded2} holds when $v$ and $w$ are linear combinations of such coefficient functions: this is similar to the argument in the proof of Proposition~\ref{p:ax+b_mainbound}.
Finally, let $v=\sum_{n=-N}^N v_n$ and $w=\sum_{m=-N}^{N} w_m$,
 where
 $v_k$ and $w_k$ are finite linear combinations of coefficient functions, associated to $\sch_k$ for $k\neq 0$ and associated to $\lambda_0$ for $k=0$.
Note that $D_\flat(v_0,\blank)=D_\flat(\blank,w_0)=0$, and so
\[
\begin{aligned}
|D_\flat(v,w)|
 &=\left|\sum_{n=-N}^N\sum_{m=-N}^N D_\flat(v_n,w_m)\right|\\
 &\leq  \sum_{n=1}^N  |D_\flat(v_n,w_{-n})|+ \sum_{n=1}^N  |D_\flat(v_{-n},w_{n})|\\
 & \leq  \sum_{n=1}^N \norm{v_n}_{\FA}\norm{w_{-n}}_{\FA}+\sum_{n=1}^N \norm{v_{-n}}_{\FA}\norm{w_{n}}_{\FA}\\
 &\leq  \left(\sum_{n=-N}^N  \norm{v_n}_{\FA}\right)\left(\sum_{n=-N}^N  \norm{w_n}_{\FA}\right)
 & =   \norm{v}_{\FA}\norm{w}_{\FA},
\end{aligned}
\]
where we used the $\ell^1$-decomposition of $A(\bbH_r)$ obtained in Proposition \ref{p:A(Hr)-decomp}. Thus condition \ref{li:bounded2} holds for all $v,w\in V$, and we may apply Lemma~\ref{l:fussy-extension}, which completes the proof.
\end{proof}


\begin{section}{Closing remarks}
In~\cite{ForRun_amenAG}, B. E. Forrest and V. Runde posed the explicit conjecture that if $G$ is a locally compact group, then $\FA(G)$ is weakly amenable if and only if the connected component of the identity in $G$ is abelian. (The conjecture may have been posed informally on earlier occasions;
there is a good account of the background context, and some partial results, in Section 2 of the survey article \cite{Spronk_BA09survey}.) 
In~particular, if $G$ is a connected non-abelian Lie group, then its Fourier algebra should not be weakly amenable.
Theorems~\ref{t:reduced-H_non-WA} and \ref{t:should-be-solvable} provide further supporting evidence for this weaker conjecture. Natural examples to try next are: the motion group of the plane, $\Real^2\rtimes\SO_2(\Real)$, and its covering groups; and the ``full'' Heisenberg group~$\bbH$. These examples are Type I groups, so one can try to analyse their Fourier algebras in terms of their irreducible representations.

We finish the paper with remarks on a possible alternative approach to our results. With hindsight, our calculations for the $ax+b$ group and the reduced Heisenberg group can be interpreted in terms of the Plancherel formulas for those groups, since the key to both constructions is an estimate of expressions of the form $\pair{\partial f}{g}_{L^2(G)}$ for an appropriately defined differential operator~$\partial$.

The Plancherel formula for $\bbH_r$ takes the form
\[ \pair{f}{g}_{L^2(\bbH_r)} =
 \int_{\Real^2} {\mathbb E}_{\Torus}f\ \overline{ {\mathbb E}_{\Torus}g }
 +  \sum_{n\in\Zahl\setminus\{0\}} |n| \Tr( \sch_n(f)\sch_n(g)^* ) \]
where ${\mathbb E}_{\Torus} : L^2(\bbH_r)\to L^2(\Real^2)$ averages along the embedded copy of $\Torus$.
(We did not find this identity stated explicitly in the sources we consulted, but it can be proved by adapting the standard proofs of the Plancherel formula for the \emph{full} Heisenberg group.)
Taking this for granted, and using the same orthogonality techniques as in Section~\ref{s:reduced-H}, some work shows that
\[ \norm{f}_{\FA(\bbH_r)} = \norm{ {\mathbb E}_{\Torus} f}_{\FA(\Real^2)} + \sum_{n\in\Zahl\setminus\{0\}} |n| \norm{\sch_n(f)}_1 \,\]
in close analogy with known formulas for the Fourier algebras of \emph{compact} groups.
If $\partial_\theta$ is as defined in Section~\ref{s:reduced-H}, then by arguing as in Remark~\ref{r:use-Plancherel-for-compact}, one can show that
\[ \abs{ \int_{\bbH_r} \partial_\theta(f)g  } \leq \sum_{n\in\Zahl\setminus\{0\}} |n|^2 \norm{\sch_n(f)\sch_n(g)^* }_1 \leq \norm{f}_{\FA(\bbH_r)}  \norm{g}_{\FA(\bbH_r)} \]
provided $f$ and $g$ are sufficiently well-behaved. Extending by continuity then gives us the desired derivation.

There is also a Plancherel formula for the $ax+b$ group, but because of non-unimodularity the Fourier transform no longer defines an isometry from $L^2(G)$ to a space of Hilbert--Schmidt operators. Instead, one has to introduce a correcting factor, in the form of the unbounded operator $K$ that played a role in Section~\ref{s:ax+b}. (See \cite{eymard-terp,khalil_ax+b} for details.) Nevertheless, with sufficiently careful book-keeping, a very similar argument can be used to give an alternative proof of Theorem~\ref{t:ax+b_non-WA}.

For the two cases treated in this paper, coefficient functions seemed to provide a more direct and secure approach. On the other hand, the Plancherel perspective may be useful for examples such as the full Heisenberg group~$\bbH$, which does not have a large supply of square-integrable representations, and for which a direct analysis in terms of coefficient functions of irreducible representations seems less promising. We intend to pursue this further in forthcoming work.
\end{section}

\appendix

\begin{section}{Appendix: $\ell^1$-decompositions of two Fourier algebras}

\subsection{Proof of Proposition~\ref{p:A(ax+b)-decomp}}
We keep the same notation used in Section~\ref{s:ax+b}. In particular, we denote the connected, real $ax+b$ group by $G$.
Since $\pi_+$ and $\pi_-$ are inequivalent square-integrable unitary representations,  $\FA_{\pi_+}\oplus \FA_{\pi_-}$ is an $\ell^1$-direct sum inside $\FA(G)$. (See Proposition 3.12 and Corollary 3.13 of \cite{Arsac},
or \cite[Proposition 3.5.18]{Zwarich_MSc}.) In particular, it is a closed subspace, and so it suffices to prove that it is also a dense subspace.

\para{Step 1} Let $\cV_0 = \lin\{\xi*_{\pi_{\pm}}\eta\colon \xi,\eta\in C_c(\Real_+^*)\}$. We claim that $\cV_0$ is dense in $L^2(G)$.

\begin{proof}
Identifying $L^2(G)$ with $L^2(\Real\times\Real_+^*)$, define $W: L^2(G) \to L^2(G)$ by $W F(b,a) = F(b, |b|a)$. This is an isometric isomorphism.
Now, given $\xi,\eta\in C_c(\Real_+^*)$, and regarding $\eta$ as an element of $L^2(\Real)$, direct calculation yields
\[ \begin{aligned}
(\cF\tp I)W(\eta\tp\xi)(b,a)
 & = \int_\Real e^{-2\pi ibt} W(\eta\tp \xi)(t,a) \,dt \\
 & = \int_0^\infty e^{-2\pi ibt} \eta(t) \xi(ta) \,dt
 & = (\xi*_{\pi_+} \overline{K\eta})(b,a).
\end{aligned} \]
A similar calculation shows that if $\xi,\eta\in C_c(\Real_+^*)$, and we define $\check{\eta}(t) = \eta(-t)$, then
\[ (\cF\tp I)W(\check{\eta}\tp \xi)= - \xi *_{\pi_{-}} \overline{K\eta}\;. \]
Let $\cS_0=\lin\{ \eta\tp\xi \colon \eta \in C_c(\Real_+^*)\oplus C_c(\Real_{-}^*), \xi \in C_c(\Real_+^*) \}$.
 The calculations above show that $(\cF\tp I)W(\cS_0)\subseteq \cV_0$.
Since $\cS_0$ is dense in $L^2(G)$, and since both $W$ and $\cF\tp I$ are invertible isometries, it follows that $\cV_0$ is dense in $L^2(G)$. This completes Step~1.
\end{proof}

\para{Step 2} Define $\Psi_\lm: L^2(G)\ptp \overline{L^2(G)} \to \FS(G)$ to be the continuous linear map $f\tp \overline{g} \to f*_\lm g$. We claim that $\Psi_\lm(\cV_0\tp\overline{\cV_0})\subseteq \FA_{\pi_+}\oplus\FA_{\pi_-}$. 

\begin{proof}
It suffices to prove that if  $f$ and $g$ are coefficient functions of $\pi_+$ or $\pi_{-}$, both generated by vectors in $C_c(\Real_+^*)$, then $f*_\lm g$ is a coefficient function of $\pi_{-}$ or $\pi_+$.
This will follow from the orthogonality relations for coefficient functions of $\pi_+$ and $\pi_-$: although these identities were stated in Proposition \ref{p:orthogonality-ax+b} only for vectors in $C_c^2(\Real_+^*)$, the same calculations all work for vectors in $C_c(\Real_+^*)$.

Given a unitary representation $\pi:G\to \cU(\cH_\pi)$ and $\xi,\eta\in \cH_\pi$, we have
\[ (\xi *_{\pi}\eta)(x^{-1}y) = \ip{\pi(x)^*\pi(y)\xi}{\eta} = \ip{\pi(y)\xi}{\pi(x)\eta} = (\xi*_{\pi} \pi(x)\eta)(y)
 \qquad(x,y\in G). \]
(This calculation works for any locally compact group, not just the $ax+b$ group.) Therefore, given $\xi,\xi',\eta,\eta'\in C_c(\Real_+^*)$,  the orthogonality relations for $\pi_+$ yield
\[
\begin{aligned}
\ip{ \lambda(x)(\xi*_{\pi_+}\eta) }{ \xi'*_{\pi_+}\eta' }
 & = \ip{ \xi*_{\pi_+}\pi_+(x)\eta }{ \xi'*_{\pi_+}\eta' } \\
 & = \ip{ K^{-\frac{1}{2}}\xi }{ K^{-\frac{1}{2}}\xi'} \ \overline{ \ip{ \pi_+(x)\eta }{ \eta'} }
 & = \ip{ K^{-\frac{1}{2}}\xi }{ K^{-\frac{1}{2}}\xi'} \ (\overline{\eta}*_{\pi_{-}} \overline{\eta'}) (x)\;,
\end{aligned}
\]
so that $(\xi*_{\pi_+}\eta)*_\lm ( \xi'*_{\pi_+}\eta')$ is a scalar multiple of $\overline{\eta}*_{\pi_{-}} \overline{\eta'}$. A similar calculation shows that
$(\xi*_{\pi_-}\eta)*_\lm (\xi'*_{\pi_-}\eta')$ is a scalar multiple of $\overline{\eta}*_{\pi_+} \overline{\eta'}$, and that
$\ip{ \lambda(x)(\xi*_{\pi_+}\eta) }{ \xi'*_{\pi_-}\eta' }  = 0$.
This completes Step 2.
\end{proof}

Putting Steps 1 and 2 together, and recalling that $\Psi_\lm$ has closed range $\FA(G)$, we see that elements of $\FA(G)$ can be approximated in $\FA(G)$-norm by elements of $\cV_0$. This completes the proof of Proposition~\ref{p:A(ax+b)-decomp}.
\hfill$\Box$

\subsection{Proof of Proposition~\ref{p:A(Hr)-decomp}}
The following argument is due to the referee. It is similar to the proof of~\cite[Theorem 3.1]{CFS}.

First recall that the Haar measure $\mu$ of $\bbH_r$ is
the product of the Haar measures of $\Real^2$ and $\Torus$. Therefore we can decompose the Hilbert space $L^2(\bbH_r)$ as
\[ L^2(\bbH_r)\simeq L^2(\Real^2)\otimes_2(\ell^2\!-\!\bigoplus_{n\in\Zahl}\Cplx\chi_n)\simeq \ell^2\!-\!\bigoplus_{n\in\Zahl}L^2(\Real ^2)\otimes_2\Cplx\chi_n\ . \]
For each $n\in \Zahl$, let $\cH_n$ denote the subspace of $L^2(\bbH_r)$ which corresponds to $L^2(\Real ^2)\otimes_2\Cplx\chi_n$ via the above isomorphisms. Note that $\cH_n$ is invariant under the left regular representation $\lambda_{\bbH_r}$. Indeed, for $f\in L^2(\Real^2)$, we have:
\[ \begin{aligned}
\lambda(x,y,e^{2\pi i \theta})(f\otimes \chi_n)(x',y',e^{2\pi i \theta'})
 & = (f\otimes\chi_n)(-x+x',-y+y',e^{2\pi i (\theta'-\theta)}e^{\pi i (-xy'+x'y)})\\
 & = e^{-2\pi in\theta}e^{\pi i n(-xy'+x'y)}f(-x+x',-y+y')\chi_n(e^{2\pi i \theta'}).
\end{aligned} \]
 
Let $\pi_n$ denote the restriction of $\lambda_{\bbH_r}$ to $\cH_n$. Observe that $\pi_n(0,0,e^{2\pi i \theta})=e^{-2\pi i n \theta}I$ whenever $n\neq 0$; so by the Stone--von Neumann theorem $\pi_n$ is unitarily equivalent to  a direct sum of copies of $\sch_{-n}$.
Moreover, it is easy to see that the restriction of $\lambda_{\bbH_r}$ to $\cH_0$ coincides with $\lambda_0\defeq \lambda_{\Real^2}\circ Q$, where $\lambda_{\Real^2}$ is the left regular representation of $\Real^2$ and $Q:\bbH_r\rightarrow \bbH_r/{\mathcal Z}\iso \Real ^2$ is the natural quotient map.

Thus, if we let $\sch$ denote the direct sum of the Schr\"odinger representations for $n\neq 0$, we have shown that $\lambda_{\bbH_r}$ is quasi-equivalent to $ \lambda_0\oplus \sch$. These two representations are pairwise disjoint, in the sense that there are no non-zero intertwining maps between the representation spaces. As before, the results of Arsac (see Corollary 3.13 of \cite{Arsac} or \cite[Proposition 3.5.18]{Zwarich_MSc}) now yield
 the decomposition
\[ \FA(\bbH_r) = \FA_{\lambda_0}(\bbH_r) \oplus_1 \FA_{\sch}(\bbH_r), \]
as required.
 \hfill$\Box$

\end{section}


\bibliographystyle{siam}
\bibliography{WAAGbib}


\para{Original affiliation}\

\noindent Yemon Choi and Mahya Ghandehari

\noindent
Department of Mathematics and Statistics\\
McLean Hall, University of Saskatchewan\\
Saskatoon (SK), Canada S7N 5E6

\para{Current address for YC}\

\noindent
Department of Mathematics and Statistics\\
Fylde College, Lancaster University\\
Lancaster, United Kingdom LA1 4YF

\noindent
Email: \texttt{y.choi1@lancaster.ac.uk}

\para{Current address for MG}\

\noindent
The Fields Institute for Research in Mathematical Sciences\\
222 College Street\\
Toronto (ON), Canada M5T 3J1

\noindent
Email: \texttt{mghandeh@fields.utoronto.ca}

\end{document}